\documentclass[12pt,reqno]{amsart}
\usepackage{fullpage,url,amssymb,enumerate}
\usepackage[all]{xy} 
\usepackage{mathrsfs} 
\usepackage{bm} 

\usepackage[OT2,T1]{fontenc}
\DeclareSymbolFont{cyrletters}{OT2}{wncyr}{m}{n}
\DeclareMathSymbol{\Sha}{\mathalpha}{cyrletters}{"58}

\usepackage{color}

\newcommand{\Adual}{\widehat{A}}

\newcommand{\WQ}{{\textup{W}}Q}

\newcommand{\Poincare}{\mathscr{P}}

\bmdefine\boldmu{\mu}

\newcommand{\Exp}{\mathbb{E}}

\newcommand{\defi}[1]{\textsf{#1}} 

\newcommand{\C}{{\mathbb C}}
\newcommand{\F}{{\mathbb F}}
\newcommand{\G}{{\mathbb G}}

\newcommand{\Q}{{\mathbb Q}}
\newcommand{\R}{{\mathbb R}}
\newcommand{\Z}{{\mathbb Z}}
\newcommand{\Qbar}{{\overline{\Q}}}

\newcommand{\kbar}{{\overline{k}}}

\newcommand{\calG}{{\mathcal G}}
\newcommand{\calH}{{\mathcal H}}
\newcommand{\calI}{{\mathcal I}}

\newcommand{\calO}{{\mathcal O}}

\newcommand{\calS}{{\mathcal S}}
\newcommand{\calT}{{\mathcal T}}

\newcommand{\calW}{{\mathcal W}}
\newcommand{\calX}{{\mathcal X}}

\newcommand{\LL}{{\mathscr L}}
\newcommand{\OO}{{\mathscr O}}

\DeclareMathOperator{\Prob}{Prob}

\DeclareMathOperator{\rk}{rk}

\DeclareMathOperator{\Char}{char}
\DeclareMathOperator{\inv}{inv}

\DeclareMathOperator{\im}{im}

\DeclareMathOperator{\Hom}{Hom}

\DeclareMathOperator{\Aut}{Aut}
\DeclareMathOperator{\Gal}{Gal}

\DeclareMathOperator{\Sel}{Sel}

\DeclareMathOperator{\Pic}{Pic}
\DeclareMathOperator{\Jac}{Jac}

\DeclareMathOperator{\Spec}{Spec}

\newcommand{\conts}{{\operatorname{conts}}}

\newcommand{\tors}{{\operatorname{tors}}}

\newcommand{\unr}{{\operatorname{unr}}}
\newcommand{\tH}{{\operatorname{th}}}

\newcommand{\HH}{{\operatorname{H}}}

\newcommand{\Sp}{\operatorname{Sp}}
\newcommand{\GSp}{\operatorname{GSp}}
\newcommand{\GL}{\operatorname{GL}}

\newcommand{\injects}{\hookrightarrow}
\newcommand{\isom}{\simeq}

\newcommand{\Intersection}{\bigcap} 
\newcommand{\intersect}{\cap} 
\newcommand{\Union}{\bigcup} 
\newcommand{\union}{\cup} 
\newcommand{\tensor}{\otimes} 
\newcommand{\directsum}{\oplus} 
\newcommand{\Directsum}{\bigoplus} 

\newtheorem{theorem}{Theorem}[section]
\newtheorem{lemma}[theorem]{Lemma}
\newtheorem{corollary}[theorem]{Corollary}
\newtheorem{proposition}[theorem]{Proposition}
\newtheorem{conjecture}[theorem]{Conjecture}

\theoremstyle{definition}
\newtheorem{definition}[theorem]{Definition}

\newtheorem{example}[theorem]{Example}

\theoremstyle{remark}

\newtheorem{remark}[theorem]{Remark}

\usepackage[
	backref,
	pdfauthor={Bjorn Poonen}, 
]{hyperref}
\usepackage[alphabetic,backrefs,lite]{amsrefs} 

\begin{document}

\title{Random maximal isotropic subspaces and Selmer groups}
\subjclass[2010]{Primary 11G10; Secondary 11G05, 11G30, 14G25, 14K15}
\keywords{Selmer group, Shafarevich-Tate group, maximal isotropic, quadratic space, Weil pairing, theta characteristic}
\author{Bjorn Poonen}
\thanks{B.~P.\ was partially supported by NSF grant DMS-0841321.}
\address{Department of Mathematics, Massachusetts Institute of Technology, Cambridge, MA 02139-4307, USA}
\email{poonen@math.mit.edu}
\urladdr{http://math.mit.edu/~poonen/}
\author{Eric Rains}
\address{Department of Mathematics, California Institute of Technology, Pasadena, CA 91125}
\email{rains@caltech.edu}
\date{April 10, 2011}

\begin{abstract}
Under suitable hypotheses, 
we construct a probability measure on the set of closed maximal 
isotropic subspaces of a locally compact quadratic space over $\F_p$.
A random subspace chosen with respect to this measure is discrete
with probability $1$,
and the dimension of its intersection with a fixed compact open
maximal isotropic subspace is a certain nonnegative-integer-valued 
random variable.

We then prove that the $p$-Selmer group of an elliptic curve is naturally
the intersection of a discrete maximal isotropic subspace with
a compact open maximal isotropic subspace 
in a locally compact quadratic space over $\F_p$.
By modeling the first subspace as being random,
we can explain the known phenomena regarding
distribution of Selmer ranks, such as the theorems of Heath-Brown,
Swinnerton-Dyer, and Kane for $2$-Selmer groups in certain families of
quadratic twists, and the average size of $2$- and $3$-Selmer groups
as computed by Bhargava and Shankar.  
Our model is compatible
with Delaunay's heuristics for $p$-torsion in Shafarevich-Tate groups,
and predicts that the average rank of elliptic curves
over a fixed number field is at most $1/2$.
Many of our results generalize to abelian varieties over global fields.
\end{abstract}

\maketitle

\section{Introduction}
\label{S:introduction}

\subsection{Selmer groups}
\label{S:introduction Selmer groups}

D.~R.~Heath-Brown \cites{Heath-Brown1993,Heath-Brown1994},
P.~Swinnerton-Dyer \cite{Swinnerton-Dyer2008},
and D.~Kane \cite{Kane-preprint}
obtained the distribution for the nonnegative integer $s(E)$
defined as the $\F_2$-dimension of the $2$-Selmer group $\Sel_2 E$
minus the dimension of the rational $2$-torsion group $E(\Q)[2]$,
as $E$ varies over quadratic twists of certain elliptic curves over $\Q$.
The distribution was the one for which
\[
	\Prob\left(s(E) = d \right)
	= \left( \prod_{j \ge 0} (1+2^{-j})^{-1} \right) 
	\left( \prod_{j=1}^d \frac{2}{2^j-1} \right).
\]
In \cite{Heath-Brown1994}, it was reconstructed from 
the moments of $2^{s(E)}$;
in \cite{Swinnerton-Dyer2008} and \cite{Kane-preprint},
it arose as the stationary distribution for a Markov process.

Our work begins with the observation that this distribution 
coincides with a distribution arising naturally in combinatorics,
namely, the limit as $n \to \infty$ 
of the distribution of $\hbox{$\dim(Z \intersect W)$}$
where $Z$ and $W$ are random maximal isotropic subspaces
inside a hyperbolic quadratic space of dimension $2n$ over $\F_2$.
Here it is essential that the maximal isotropic subspaces 
be isotropic not only for the associated symmetric bilinear pairing,
but also for the quadratic form; otherwise, one would obtain the wrong
distribution.
That such a quadratic space might be relevant is suggested
already by the combinatorial calculations in \cite{Heath-Brown1994}.

Is it just a coincidence, 
or is there some direct relation between Selmer groups and 
intersections of maximal isotropic subgroups?
Our answer is that $\Sel_2 E$ is naturally the 
intersection of two maximal isotropic subspaces
in an {\em infinite-dimensional} quadratic space $V$ over $\F_2$.
The fact that it could be obtained as an intersection
of two subspaces that were maximal isotropic 
for a pairing induced by a Weil pairing 
is implicit in standard arithmetic duality theorems.

To make sense of our answer, 
we use the theory of quadratic forms on locally compact abelian groups
as introduced by A.~Weil in~\cite{Weil1964}.
The locally compact abelian group $V$ in the application
is the restricted direct product of the groups $\HH^1(\Q_p,E[2])$
for $p \le \infty$
with respect to the subgroups of unramified classes.
The quadratic form $Q$ is built using D.~Mumford's Heisenberg group,
using ideas of Yu.~Zarhin~\cite{Zarhin1974}*{\S2}.
Then arithmetic duality theorems are applied to show
that the images of the compact group $\prod_{p \le \infty} E(\Q_p)/2E(\Q_p)$
and the discrete group $\HH^1(\Q,E[2])$ are maximal isotropic in $(V,Q)$.
Their intersection is $\Sel_2 E$.

\subsection{Conjectures for elliptic curves}

This understanding of the structure of $\Sel_2 E$ suggests the following, 
in which we replace $2$ by $p$ and generalize to global fields:
\begin{conjecture}
\label{C:conjecture for Sel_p E}
Fix a global field $k$ and a prime $p$.
  \begin{enumerate}[\upshape (a)]
  \item \label{I:distribution for Sel_p E}
As $E$ varies over all elliptic curves over $k$,
\[
	\Prob\left(\dim_{\F_p} \Sel_p E = d \right)
	= \left( \prod_{j \ge 0} (1+p^{-j})^{-1} \right) 
	\left( \prod_{j=1}^d \frac{p}{p^j-1} \right).
\]
(For the sake of definiteness, we define the probability
by considering the finitely many elliptic curves $y^2=x^3+ax+b$ 
where $a,b \in k$ have height $\le B$, and looking at the limit
of the probability as $B \to \infty$; use a long Weierstrass equation
if $\Char k$ is $2$ or $3$.)
\item \label{I:average size for Sel_p E}
The average of $\#\Sel_p E$ over all $E/k$ is $p+1$.
\item \label{I:moments for Sel_p E}
For $m \in \Z_{\ge 0}$, 
the average of $(\#\Sel_p E)^m$ over all $E/k$ is $\prod_{i=1}^m (p^i+1)$.
  \end{enumerate}
\end{conjecture}

Several results in the direction of Conjecture~\ref{C:conjecture for Sel_p E}
are known:
\begin{itemize}
\item When $p=2$, 
Heath-Brown~\cite{Heath-Brown1994}*{Theorem~2}
proved the analogue of Conjecture~\ref{C:conjecture for Sel_p E}
for the family of quadratic twists 
of $y^2=x^3-x$ over $\Q$,
with the caveat that the distribution of $\dim \Sel_2 E$
is shifted by $+2$ to account for the contribution of $E[2]$ 
(cf.~Remark~\ref{R:adjustment}).
Swinnerton-Dyer~\cite{Swinnerton-Dyer2008} and 
Kane~\cite{Kane-preprint} generalized this result
to the family of quadratic twists
of any fixed elliptic curve $E/\Q$ with $E[2] \subseteq E(\Q)$
and no rational cyclic subgroup of order $4$.
\item For the family of all elliptic curves $E/\Q$
with $E[2] \subset E(\Q)$, G.~Yu~\cite{Yu2006}*{Theorem~1}
built upon Heath-Brown's approach
to prove that the average size of $\Sel_2 E$ is finite.
(Strictly speaking, if the limit defining the average does not exist,
the result holds with a lim sup.)
See also~\cite{Yu2005} for results for other families of elliptic
curves over $\Q$.
\item For the family of all elliptic curves over $\F_q(t)$ with $3 \nmid q$,
A.~J.~de~Jong~\cite{DeJong2002}
proved that the average size of $\Sel_3 E$ (in the lim sup sense) 
is at most $4+O(1/q)$,
where the $O(1/q)$ is an explicit rational function of $q$.
In fact, de~Jong speculated that the truth was $4$,
and that the same might hold for number fields.
\item For the family of all elliptic curves over $\Q$,
M.~Bhargava and A.~Shankar proved that
the average size of $\Sel_2$ is $3$ \cite{Bhargava-Shankar-preprint1} 
and 
the average size of $\Sel_3$ is $4$ \cite{Bhargava-Shankar-preprint2}.
\item 
For $E$ over a number field $k$ with a real embedding and with $E[2](k)=0$,
B.~Mazur and K.~Rubin~\cite{Mazur-Rubin2010} and Z.~Klagsbrun
showed how to twist judiciously
to obtain lower bounds on the number of 
quadratic twists of $E$ (up to a bound)
with prescribed $\dim \Sel_2$.
\end{itemize}

As has been observed by Bhargava,
since $\rk E(k) \le \dim_{\F_p} \Sel_p E$,
Conjecture~\ref{C:conjecture for Sel_p E}\eqref{I:average size for Sel_p E}
would imply that $\Prob(\rk E(k) \ge 2)$ is at most $(p+1)/p^2$.
If we assume this for an infinite sequence of primes $p$,
we conclude that asymptotically $100\%$ of elliptic curves over $k$
have rank $0$ or~$1$.

A priori, the average rank could still be greater than $1$ if
there were rare curves of very high rank,
but
Conjecture~\ref{C:conjecture for Sel_p E}\eqref{I:average size for Sel_p E}
for an infinite sequence of primes $p$ 
implies also that the elliptic curves of rank $\ge 2$
contribute nothing to the average value of $\rk E(k)$,
and in fact nothing to the average value of $p^{\rk E(k)}$.
(Proof: Define $e_p$ as the lim sup as $B \to \infty$ 
of the sum of $p^{\rk E(k)}$ over curves of rank $\ge 2$
of height bounded by $B$
divided by the total number of curves of height bounded by $B$.
If $q$ is a larger prime, then $e_p \le (p^2/q^2) e_q \le (p^2/q^2)(q+1)$,
which tends to $0$ as $q \to \infty$, so $e_p=0$.)

If in addition we assume that the parity of $\rk E(k)$ is equidistributed,
we obtain the following well-known conjecture:

\begin{conjecture}
\label{C:folklore}
Fix a global field $k$.
Asymptotically $50\%$ of elliptic curves over $k$ have rank~$0$, 
and $50\%$ have rank~$1$.
Moreover, the average rank is $1/2$.
\end{conjecture}

\begin{remark}
D.~Goldfeld conjectured that the average rank in a family of quadratic twists
of a fixed elliptic curve over $\Q$ 
was $1/2$~\cite{Goldfeld1979}*{Conjecture~B}.
Other evidence for Conjecture~\ref{C:folklore} 
was provided by the extensive study by N.~Katz and P.~Sarnak of 
a function field analogue~\cites{Katz-Sarnak1999a,Katz-Sarnak1999b}.
\end{remark}

Also, as has been observed by Rubin, the distribution in
Conjecture~\ref{C:conjecture for Sel_p E}\eqref{I:distribution for Sel_p E}
tends, as $p \to \infty$,
to the distribution assigning probability $50\%$ to each of $0$ and $1$.
Thus, even without assuming equidistribution of parity, 
Conjecture~\ref{C:conjecture for Sel_p E} for any infinite set of primes $p$
would imply not only that $100\%$ of elliptic curves have rank $0$ or $1$,
but also that at least $50\%$ have rank $0$,
and that the average rank is {\em at most} $1/2$.

Conjecture~\ref{C:conjecture for Sel_p E}\eqref{I:distribution for Sel_p E} 
for a single $p$
does not duplicate C.~Delaunay's prediction 
for $\dim_{\F_p} \Sha(E)[p]$ \cites{Delaunay2001,Delaunay2007}.
Instead the predictions complement each other:
we prove that the only distribution on $\rk E(\Q)$ 
compatible with both predictions 
is the one in Conjecture~\ref{C:folklore},
for which $\rk E(\Q)$ is $0$ or $1$, with probability $1/2$ each 
(see Theorem~\ref{T:Delaunay compatibility}).
A related result, that 
Conjectures \ref{C:conjecture for Sel_p E}\eqref{I:distribution for Sel_p E}
(for $p=2$) and~\ref{C:folklore} 
together imply the $\Sha[2]$ predictions for rank $0$ and $1$,
had been observed at the end of~\cite{Delaunay2007}.

If we also use the heuristic that the dimensions of $\dim_{\F_p} \Sel_p E$ 
for different $p$ are independent except for the constraint
that their parities are equal,
we are led to the following generalization of 
Conjecture~\ref{C:conjecture for Sel_p E}:

\begin{conjecture}
\label{C:conjecture for Sel_n E}
Fix a global field $k$ and let $n$ be a squarefree positive integer.
Let $\omega(n)$ be the number of prime factors of $n$.
  \begin{enumerate}[\upshape (a)]
  \item \label{I:distribution for Sel_n E}
Fix $d_p \in \Z_{\ge 0}$ for each prime $p$ dividing $n$.
As $E$ varies over all elliptic curves over $k$,
\[
	\Prob\left( \Sel_n E \isom \prod_{p|n} (\Z/p\Z)^{d_p} \right)
	= 2^{\omega(n)-1} \prod_{p|n} \left( \left( \prod_{j \ge 0} (1+p^{-j})^{-1} \right) 
	\left( \prod_{j=1}^{d_p} \frac{p}{p^j-1} \right) \right)
\]
if the $d_p$ all have the same parity, 
and the probability is $0$ otherwise.
\item \label{I:average size for Sel_n E}
The average of $\#\Sel_n E$ over all $E/k$ is 
the sum of the divisors of $n$.
\item \label{I:moments for Sel_n E}
For $m \in \Z_{\ge 0}$, 
the average of $(\#\Sel_n E)^m$ over all $E/k$ is 
$\prod_{p|n} \prod_{i=1}^m (p^i+1)$.
  \end{enumerate}
\end{conjecture}

The factor of $2^{\omega(n)-1}$ arises in~\eqref{I:distribution for Sel_n E},
because only $2$ of the $2^{\omega(n)}$ 
choices of parities for $p \mid n$ are constant sequences.

\begin{remark}
Based on investigations for $n \le 5$,
Bhargava and Shankar have proposed
Conjecture~\ref{C:conjecture for Sel_n E}\eqref{I:average size for Sel_n E} 
for {\em all} positive integers $n$, at least for $k=\Q$.
\end{remark}

\begin{remark}
As was noticed during a discussion with Bhargava and Kane,
if we combine Delaunay's heuristics for $\Sha[n]$
with Conjecture~\ref{C:folklore} for varying $E/\Q$, 
we can predict the distribution
for the abelian group $\Sel_n E$ 
for {\em any} fixed positive integer~$n$.
Namely, $E(\Q)_\tors=0$ with probability $1$,
and in that case the term on the left in

\[
	0 \to \frac{E(\Q)}{nE(\Q)} \to \Sel_n E \to \Sha(E)[n] \to 0
\]
is free, so the sequence of $\Z/n\Z$-modules splits;
thus, given Conjecture~\ref{C:folklore}, 
the distribution of $\Sel_n E$ can be deduced from
knowing the distribution of $\Sha(E)[n]$ 
for rank $0$ curves and for rank $1$ curves.
\end{remark}

\subsection{Conjectures for abelian varieties}

In fact, our theorems are proved in a more general context, 
with $[2] \colon E \to E$
replaced by any self-dual isogeny $\lambda \colon A \to \Adual$ 
that is of odd degree
or that comes from a symmetric line sheaf 
on an abelian variety $A$ over a global field $k$.
(See Theorem~\ref{T:Selmer is intersection}.)
In this setting, we have a surprise:
it is not $\Sel_\lambda A$ itself that is the intersection of 
maximal isotropic subgroups, but its quotient by $\Sha^1(k,A[\lambda])$,
and the latter group is sometimes nonzero,
as we explain in Section~\ref{S:Sha}.
Moreover, for certain families, such as the family of all genus~$2$ curves,
there may be ``causal'' subgroups of $\Sel_\lambda A$, 
which increase its expected size by a constant factor.
Taking these into account suggests the following:

\begin{conjecture}
\label{C:conjecture for odd degree hyperelliptic}
Fix a global field $k$ of characteristic not $2$, 
and fix a positive integer $g$.
Let $f \in k[x]$ range over separable polynomials of degree $2g+1$ 
(with coefficients of height bounded by $B$, with $B \to \infty$).
Let $X$ be the smooth projective model of $y^2=f(x)$.
Construct the Jacobian $A:=\Jac X$.
Then the analogues of 
Conjectures \ref{C:conjecture for Sel_p E}, 
\ref{C:folklore},
and~\ref{C:conjecture for Sel_n E} hold for $\Sel_p A$ and $\Sel_n A$,
with the same distributions.
They hold also with $2g+1$ replaced by $2g+2$ {\em if $n$ is odd}.
\end{conjecture}

\begin{conjecture}
\label{C:even genus}
Fix a global field $k$ of characteristic not $2$, 
and fix an {\em even} positive integer $g$.
Let $f \in k[x]$ range over polynomials of degree $2g+2$ 
(with coefficients of height bounded by $B$, with $B \to \infty$).
Let $X$ be the smooth projective model of $y^2=f(x)$,
and let $A=\Jac X$.
  \begin{enumerate}[\upshape (a)]
  \item\label{I:+1}
If $X_{\Sel_2}$ is the $\Z_{\ge 0}$-valued random variable
predicted by 
Conjecture~\ref{C:conjecture for Sel_p E}\eqref{I:distribution for Sel_p E}
to model $\dim_{\F_2} \Sel_2 E$,
then the analogous random variable for $\Sel_2 A$ is $X_{\Sel_2}+1$.
\item 
The average of $\#\Sel_2 A$ is $6$ (instead of $3$).
\item
For $m \in \Z_{\ge 0}$, 
the average of $(\#\Sel_2 A)^m$ is $2^m \prod_{i=1}^m (2^i+1)$.
\end{enumerate}
\end{conjecture}

Example~\ref{E:Sel_2 of genus 2 curve} will explain the rationale
for the $+1$ in Conjecture~\ref{C:even genus}\eqref{I:+1},
and will explain why we do not venture to make 
an analogous conjecture for odd $g$.

\begin{remark}
Although we have formulated 
conjectures for the family of all curves of a specified type,
our model makes sense also for more limited families
(e.g., a family of quadratic twists).
One should take into account systematic contributions
to the Selmer group, however,
as was necessary in Conjecture~\ref{C:even genus}.
\end{remark}

\section{Random maximal isotropic subspaces}

\subsection{Quadratic modules}
\label{S:quadratic modules}

See~\cite{Scharlau1985}*{1.\S6 and 5.\S1} for the definitions of this section.
Let $V$ and $T$ be abelian groups.
Call a function $Q \colon V \to T$ 
a ($T$-valued) \defi{quadratic form} if 
$Q$ is a quadratic map (i.e., 
the symmetric pairing $\langle\;,\;\rangle \colon V \times V \to T$
sending $(x,y)$ to $Q(x+y)-Q(x)-Q(y)$ is bilinear)
and $Q(av)=a^2 Q(v)$ for every $a \in \Z$ and $v \in V$.
Then $(V,Q)$ is called a \defi{quadratic module}.

\begin{remark}
\label{R:quadratic map vs quadratic form}
A quadratic map $Q$ satisfying the identity $Q(-v)=Q(v)$ is a quadratic form.
(Taking $x=y=0$ shows that $Q(0)=0$, 
and then $Q(av+(-v))-Q(av)-Q(-v)=a(Q(0)-Q(v)-Q(-v))$
computes $Q(av)$ for other $a \in \Z$ by induction.)
\end{remark}

\begin{lemma}
\label{L:killed by l or 2l}
Let $(V,Q)$ be a quadratic module.
Suppose that $v \in V$ and $\ell \in \Z$ are such that $\ell v=0$.
If $\ell$ is odd, then $\ell Q(v)=0$.
If $\ell$ is even, then $2\ell Q(v)=0$.
\end{lemma}

\begin{proof}
We have $\ell^2 Q(v) = Q(\ell v) = 0$,
and $2\ell Q(v) = \ell \langle v,v \rangle = \langle \ell v,v \rangle = 0$.
\end{proof}

Given a subgroup $W \subseteq V$, let 
$W^\perp := \{v \in V : \langle v,w \rangle = 0 \textup{ for all $w \in W$}\}$.
Call $W$ a \defi{maximal isotropic subgroup} of $(V,Q)$
if $W^\perp=W$ and $Q|_W=0$.
Let $\calI_V$ be the set of maximal isotropic subgroups of $(V,Q)$.

\begin{remark}
Say that $W$ is maximal isotropic {\em for the pairing} $\langle\;,\;\rangle$
if $W^\perp=W$.
If $W=2W$ or $T[2]=0$, then $W^\perp=W$ implies $Q|_W=0$,
but in general $Q|_W=0$ is a nonvacuous extra condition.
\end{remark}

Call a quadratic module $(V,Q)$ \defi{nondegenerate} 
if $Q$ is $\R/\Z$-valued and $V$ is finite 
(we will relax this condition in Section~\ref{S:locally compact})
and the homomorphism $V \to V^*:=\Hom(V,\R/\Z)$ defined by 
$v \mapsto (w \mapsto \langle v,w \rangle)$ is an isomorphism.
Call $(V,Q)$ \defi{weakly metabolic} 
if it is nondegenerate and contains a maximal isotropic subgroup.
(\defi{Metabolic} entails the additional condition that the subgroup
be a direct summand.)

\begin{remark}
\label{R:subquotient}
Suppose that $(V,Q)$ is a nondegenerate quadratic module,
and $X$ is an isotropic subgroup of $(V,Q)$.
Then
\begin{enumerate}[\upshape (a)]
\item\label{I:subquotient is quadratic module}
The quotient $X^\perp/X$ is a nondegenerate quadratic module
under the quadratic form $Q_X$ induced by $Q$.
\item\label{I:subquotient pushing subgroups}
If $W \in \calI_V$,
then $(W \intersect X^\perp)+X \in \calI_V$, 
and $((W \intersect X^\perp) + X)/X \in \calI_{X^\perp/X}$.
Let $\pi^{V,X^\perp/X}(W)$ denote this last subgroup,
which is the image of $W \intersect X^\perp$ in $X^\perp/X$.
\end{enumerate}
\end{remark}

\begin{remark}
\label{R:Witt group}
If $(V,Q)$ is a nondegenerate quadratic module with $\#V <\infty$,
the obstruction to $V$ being weakly metabolic
is measured by an abelian group $\WQ \isom \Directsum_p \WQ(p)$
called the 
\defi{Witt group of nondegenerate quadratic forms on finite abelian groups} 
\cite{Scharlau1985}*{5.\S1}.
The obstruction for $(V,Q)$ equals the obstruction for 
$X^\perp/X$ for any isotropic subgroup $X$ of $(V,Q)$ 
(cf.~\cite{Scharlau1985}*{Lemma~5.1.3}).
\end{remark}

\subsection{Counting subspaces}

\begin{proposition}
\label{P:counting}
Let $(V,Q)$ be a $2n$-dimensional 
weakly metabolic quadratic space over $F:=\F_p$,
with $Q$ taking values in $\frac{1}{p}\Z/\Z \isom F$.
\begin{enumerate}[\upshape (a)]
\item\label{I:equal size fibers}
All fibers of $\pi^{V,X^\perp/X} \colon \calI_V \to \calI_{X^\perp/X}$ 
have size $\prod_{i=1}^{\dim X} (p^{n-i}+1)$.
\item \label{I:number of maximal isotropic subspaces}
We have $\#\calI_V = \prod_{j=0}^{n-1} (p^j+1)$.
\item \label{I:probability of given intersection dim}
Let $W$ be a fixed maximal isotropic subspace of $V$.
Let $X_n$ be the random variable $\dim(Z \intersect W)$,
where $Z$ is chosen uniformly at random from $\calI_V$.
Then $X_n$ is a sum of independent Bernoulli random variables $B_1,\ldots,B_n$
where $B_i$ is $1$ with probability $1/(p^{i-1}+1)$ and $0$ otherwise.
\item\label{I:generating function}
For $0 \le d \le n$, let $a_{d,n}:=\Prob(X_n=d)$,
and let $a_d:=\lim_{n \to \infty} a_{d,n}$.
Then 
\begin{align*}
	\sum_{d \ge 0} a_{d,n} z^d 
	&= \prod_{i=0}^{n-1} \frac{z+p^i}{1+p^i} 
		= \prod_{i=0}^{n-1} \frac{1 + p^{-i} z}{1 + p^{-i}}. \\
	\sum_{d \ge 0} a_d z^d 
		&= \prod_{i=0}^\infty \frac{1 + p^{-i} z}{1 + p^{-i}}. \\
\end{align*}
\item\label{I:adn}
For $0 \le d \le n$, we have
\[
	a_{d,n} = \prod_{j=0}^{n-1} (1+p^{-j})^{-1}
		\prod_{j=1}^d \frac{p}{p^j-1}
		\prod_{j=0}^{d-1} (1-p^{j-n}).
\]
\item\label{I:ad}
For $d \ge 0$, we have
\[
	a_d = c \prod_{j=1}^d \frac{p}{p^j-1},
\]
where 
\[
	c := \prod_{j \ge 0} (1+p^{-j})^{-1} 
	= \frac12 \prod_{i \ge 0} (1-p^{-(2i+1)}).
\]
\end{enumerate}
\end{proposition}

\begin{proof}\hfill
\begin{enumerate}[\upshape (a)]
\item
Choose a full flag in $X$;
then $\pi^{V,X^\perp/X}$ factors into $\dim X$ maps of the same type, 
so we reduce to the case $\dim X=1$.
Write $X=Fv$ with $v \in V$.
For $Z \in \calI_V$, let $\overline{Z}$ be its image in $\calI_{X^\perp/X}$.
There is a bijection
$\{Z \in \calI_V : v \in Z\} \to \calI_{X^\perp/X}$
defined by $Z \mapsto \overline{Z} = Z/X$.

Fix $\overline{W} \in \calI_{X^\perp/X}$, 
and let $W \in \calI_V$ be such that $\overline{W}=W/X$.
We want to show that 
$\#\{Z \in \calI_V:\overline{Z}=\overline{W}\} = p^{n-1}+1$.
This follows once we show that the map
\begin{align*}
	\{Z \in \calI_V: \overline{Z}=\overline{W}\} 
	&\to \{\textup{codimension $1$ subspaces of $W$}\} \union \{W\} \\
	Z &\mapsto Z \intersect W
\end{align*}
is a bijection.
If $v \in Z$, then $\overline{Z}=\overline{W}$ implies that $Z=W$.
If $v \notin Z$,
then $\overline{Z}=\overline{W}$ implies that 
$Z \intersect W$ has codimension $1$ in $W$.
Conversely, for a given $W_1$ of codimension $1$ in $W$,
the $Z \in \calI_V$ containing $W_1$
are in bijection with the maximal isotropic subspaces
of the weakly metabolic $2$-dimensional space $W_1^\perp/W_1$,
which is isomorphic to $(F^2,xy)$, so there are two such $Z$:
one of them is $W$, and the other satisfies $Z \intersect W = W_1$
and $\overline{Z}=\overline{W}$.
Thus we have the bijection.
\item
Apply~\eqref{I:equal size fibers} to a maximal isotropic $X$.
\item 
If $n>0$, fix a nonzero $v$ in $W$, and define $\overline{Z}$
as in the proof of~\eqref{I:equal size fibers}.
Then
\[
	\dim Z \intersect W = \dim \overline{Z} \intersect \overline{W}
		+ \delta_{v \in Z},
\]
where $\delta_{v\in W}$ is $1$ if $v \in Z$ and $0$ otherwise.
The term $\dim \overline{Z} \intersect \overline{W}$ 
has the distribution $X_{n-1}$.
Conditioned on the value of $\overline{Z}$,
the term $\delta_{v \in Z}$ is $1$ with probability $1/(p^{n-1}+1)$
and $0$ otherwise, since there are $p^{n-1}+1$ subspaces $Z \in \calI_V$
with the given $\overline{Z}$,
and only one of them 
(namely, the preimage of $\overline{Z}$ under $V \to V/Fv$) 
contains $v$.
Thus $X_n$ is the sum of $X_{n-1}$ 
and the independent Bernoulli random variable $B_n$,
so we are done by induction on $n$.
\item
The generating function for $X_n$ is the product of 
the generating functions for $B_1,\ldots,B_n$;
this gives the first identity.
The second follows from the first.
\item
This follows from \eqref{I:generating function} 
and Cauchy's $q$-binomial theorem
(which actually goes back to \cite{Rothe1811}
and is related to earlier formulas of Euler).
Namely, set $t=1/p$ in formula~(18) of \cite{Cauchy1843},
and divide by $\prod_{j=0}^{n-1} (1+p^{-j})$.
\item
Take the limit of~\eqref{I:adn} as $n \to \infty$.
The alternative formula for $c$ follows from substituting 
\[
	1+p^{-j} = \frac{1-p^{-2j}}{1-p^{-j}}
\]
for $j \ne 0$ and cancelling common factors.\qedhere
\end{enumerate}
\end{proof}

\begin{remark}
There is a variant for finite-dimensional vector spaces $V$ 
over a finite field $F$ of non-prime order.
One can define the notion of 
weakly metabolic quadratic form $Q \colon V \to F$,
and then prove Proposition~\ref{P:counting} with $q$ in place of $p$.

If we consider only even-dimensional
nondegenerate quadratic spaces over $F$,
then the obstruction analogous to that in Remark~\ref{R:Witt group}
takes values in a group of order $2$.
The obstruction is the discriminant in $F^\times/F^{\times 2}$
if $\Char F \ne 2$,
and the Arf invariant (see \cite{Scharlau1985}*{9.\S4})
if $\Char F = 2$.
\end{remark}

\begin{remark}
\label{R:4-torsion}
By Lemma~\ref{L:killed by l or 2l},
a quadratic form on a $2$-torsion module will in general 
take values in the {\em 4-torsion} of the image group.
Thus we need an analogue of 
Proposition~\ref{P:counting}\eqref{I:probability of given intersection dim}
for a $\frac14 \Z/\Z$-valued quadratic form $Q$ on a $2n$-dimensional 
$\F_2$-vector space $V$ such that $Q(V) \not\subset \frac12 \Z/\Z$,
or equivalently such that $\langle x,x \rangle = 2 Q(x)$ 
is not identically $0$.

The map $x \mapsto \langle x,x\rangle$ 
is a linear functional $V \to \frac12 \Z/\Z \isom \F_2$
since
\[
	\langle x+y,x+y \rangle 
	= \langle x,x \rangle + 2 \langle x,y \rangle + \langle y,y \rangle
	= \langle x,x \rangle + \langle y,y \rangle.
\]
Hence there exists a nonzero $c \in V$ such that 
$\langle x,x \rangle = \langle x,c \rangle$ for all $x \in V$.
This equation shows that
for any maximal isotropic subspace $W$ of $V$,
we have $c \in W^\perp = W$.
The map $W \mapsto W/\F_2 c$
defines a bijection between
the set of maximal isotropic subspaces of $V$ 
and the set of maximal isotropic subspaces of $(\F_2 c)^\perp/\F_2 c$,
which is a $\frac12 \Z/\Z$-valued quadratic space.
So the random variable $\dim(Z \intersect W)$ for $V$
is $1$ plus the corresponding random variable 
for the $\frac12 \Z/\Z$-valued quadratic space of dimension $\dim V - 2$.
\end{remark}

\begin{definition}
\label{D:X_Sel}
Given a prime $p$,
let $X_{\Sel_p}$ be a $\Z_{\ge 0}$-valued random variable
such that for any $d \in \Z_{\ge 0}$,
the probability $\Prob(X_{\Sel_p}=d)$ equals the $a_d$ in 
Proposition~\ref{P:counting}\eqref{I:generating function}.
\end{definition}

In the notation of 
Proposition~\ref{P:counting}\eqref{I:probability of given intersection dim}, 
we can also write
\[
	X_{\Sel_p} = \lim_{n \to \infty} X_n = \sum_{n=1}^\infty B_n.
\]

\begin{remark}
The distribution of $X_{\Sel_2}$ 
agrees with the distribution of $s(E)$
mentioned at the beginning of Section~\ref{S:introduction Selmer groups}.
\end{remark}

\subsection{Some topology}
\label{S:topology}

To interpret $a_r$ as a probability and not only a limit of probabilities,
we are led to consider infinite-dimensional quadratic spaces.
The na\"{\i}ve dual of such a space $V$ is too large to be isomorphic to $V$,
so we consider spaces with a locally compact topology
and use the Pontryagin dual.
In order to define a probability measure on 
the set of maximal isotropic subspaces,
we need additional countability constraints.
This section proves the equivalence of several such countability constraints.

For a locally compact abelian group $G$,
define the \defi{Pontryagin dual} $G^* := \Hom_{\conts}(G,\R/\Z)$.
Recall that a topological space is \defi{$\sigma$-compact}
if it is expressible as a union of countably many compact subspaces,
\defi{first-countable} if each point has a countable basis of neighborhoods,
\defi{second-countable} if the topology admits a countable basis,
and \defi{separable} if it has a countable dense subset.

\begin{proposition}
\label{P:sigma-compact}
Let $G$ be a locally compact abelian group.
The following are equivalent:
\begin{enumerate}[\upshape (a)]
\item $G^*$ is $\sigma$-compact.
\item $G$ is first-countable.
\item $G$ is metrizable.
\end{enumerate}
Moreover, $G$ is second-countable 
if and only if $G$ and $G^*$ are both $\sigma$-compact.
\end{proposition}

\begin{proof}
After peeling off a direct factor $\R^n$ from $G$,
we may assume that $G$ contains a compact open subgroup $K$, by the 
Pontryagin--van Kampen structure theorem \cite{VanKampen1935}*{Theorem~2}.
Each of (a), (b), (c) holds for $G$ if and only if it holds for $K$,
and for $K$ the three conditions are equivalent to second-countability
by \cite{Kakutani1943}*{Theorem~2 and the bottom of page~366}.
To prove the final statement, observe that
$G$ is second-countable if and only if $K$ is second-countable
and $G/K$ is countable.
By the above, $K$ is second-countable if and only $G^*$
is $\sigma$-compact;
on the other hand, $G/K$ is countable if and only if $G$ is $\sigma$-compact.
\end{proof}

\begin{corollary}
\label{C:second-countable}
Let $G$ be a locally compact abelian group such that $G \isom G^*$.
Then the following are equivalent:
\begin{enumerate}[\upshape (a)]
\item $G$ is $\sigma$-compact.
\item $G$ is first-countable.
\item $G$ is metrizable.
\item $G$ is second-countable.
\item $G$ is separable.
\end{enumerate}
\end{corollary}

\begin{proof}
Proposition~\ref{P:sigma-compact} formally implies the 
equivalence of (a), (b), (c), (d).
To obtain (d)$\implies$(e), choose one point from each nonempty 
set in a countable basis.
To prove (e)$\implies$(a), reduce to the case
that $G$ contains a compact open subgroup $K$;
then separability implies that $G/K$ is finite, so $G$ is $\sigma$-compact.
\end{proof}

\subsection{Quadratic forms on locally compact abelian groups}
\label{S:locally compact}

\begin{definition}
A \defi{locally compact quadratic module} $(V,Q)$
is a locally compact abelian group $V$
equipped with a continuous quadratic form $Q \colon V \to \R/\Z$.
(This notion was introduced in~\cite{Weil1964}*{p.~145},
where $Q$ was called a ``caract\`ere du second degr\'e'';
for him, the codomain of $Q$ was the group of complex numbers of absolute
value~$1$, because he was interested in the Fourier transforms of such $Q$.)
\end{definition}

The definitions of \defi{maximal isotropic} and \defi{nondegenerate} extend
to this setting.

\begin{definition}
\label{D:weakly metabolic}
Call a nondegenerate locally compact quadratic module $(V,Q)$ 
\defi{weakly metabolic} 
if it contains a {\em compact open} maximal isotropic subgroup $W$;
we then say also that $(V,Q,W)$ is weakly metabolic.
\end{definition}

\begin{remark}
\label{R:two notions of weakly metabolic}
In Definition~\ref{D:weakly metabolic},
it would perhaps be more natural to require the subgroup $W$ to be only closed,
not necessarily compact and open.
Here we explain that the two definitions are equivalent
when $V$ contains a compact open subgroup,
which is not a strong hypothesis,
since by the Pontryagin--van Kampen structure theorem,
if $V$ is a locally compact abelian group,
then $V \isom \R^n \directsum V'$ as topological groups,
where $V'$ contains a compact open subgroup.

If $(V,Q)$ is a nondegenerate locally compact quadratic module
containing a compact open subgroup $K$,
then $X:=K \intersect K^\perp$ is an compact open subgroup that
is isotropic for the pairing.
Then $Q$ restricts to a continuous {\em linear} map $X \to \frac12 \Z/\Z$,
and its kernel $Y$ is a compact open subgroup that is isotropic for $Q$.
Next, if $W$ is {\em any} closed maximal isotropic subgroup of $(V,Q)$,
then $(W \intersect Y^\perp) +Y$ is a compact open maximal isotropic
subgroup of $(V,Q)$ 
(cf.~Remark~\ref{R:subquotient}\eqref{I:subquotient pushing subgroups}).
\end{remark}

\begin{remark}
\label{R:Witt group element for locally compact}
If $(V,Q)$ is a nondegenerate locally compact quadratic module
containing a compact open isotropic subgroup $X$, 
then the obstruction to $(V,Q)$ 
containing a maximal isotropic closed subgroup
is the same as that for $X^\perp/X$,
so the obstruction is measured by an element of $\WQ$
that is independent of $X$ (cf.~Remark~\ref{R:Witt group}).
\end{remark}

\begin{example}
\label{E:hyperbolic locally compact quadratic module}
If $W$ is a locally compact abelian group,
then $V:=W \times W^*$ may be equipped with 
the quadratic form $Q((w,f)):=f(w)$.
If $W$ contains a compact open subgroup $Y$,
then its annihilator in $W^*$ is a compact open subgroup $Y'$ of $W^*$,
and $X:=Y \times Y'$ is a compact open subgroup of $V$,
so Remark~\ref{R:two notions of weakly metabolic}
shows that $(V,Q)$ is weakly metabolic.
\end{example}

\begin{example}[cf.~\cite{Braconnier1948}*{Th\'eor\`eme~1}]
\label{E:restricted direct product}
Suppose that $(V_i,Q_i,W_i)$ for $i \in I$ are weakly metabolic.
Define the \defi{restricted direct product}
\[
	V:=\sideset{}{'}\prod_{i \in I} (V_i,W_i)
	:= \left\{ (v_i)_{i \in I} \in \prod_{i \in I} V_i : 
		v_i \in W_i \textup{ for all but finitely many $i$} \right\}.
\]
Let $W:=\prod_{i \in I} W_i$.
As usual, equip $V$ with the topology 
for which $W$ is open and has the product topology.
For $v:=(v_i) \in V$, define $Q(v) = \sum_{i \in I} Q_i(v_i)$,
which makes sense since $Q_i(v_i)=0$ for all but finitely many $i$.
Then $(V,Q,W)$ is another weakly metabolic locally compact quadratic module.

Moreover,
if $I$ is countable and each $V_i$ is second-countable,
then $V$ is second-countable too.
({\em Proof:} Use Corollary~\ref{C:second-countable} 
to replace second-countable by $\sigma$-compact.
If each $V_i$ is $\sigma$-compact, then each $V_i/W_i$ is countable, 
so $V/W \isom \Directsum_{i \in I} V_i/W_i$ is countable,
so $V$ is $\sigma$-compact.)
\end{example}

Let $(V,Q)$ be a locally compact quadratic module.
Let $\calI_V$ be the set of maximal isotropic closed subgroups of $(V,Q)$.
Let $\calX_V$ be the poset of compact open isotropic subgroups of $(V,Q)$,
ordered by (reverse) inclusion.

\begin{theorem}
\label{T:probability measure for LCA group}
Let $(V,Q,W)$ be a 
second-countable weakly metabolic locally compact quadratic module.
\begin{enumerate}[\upshape (a)]
\item
The set $\calX_V$ is a countable directed poset.
\item\label{I:inverse system}
The finite sets $\calI_{X^\perp/X}$ for $X \in \calX_V$ 
with the maps $\pi^{X_1^\perp/X_1,X_2^\perp/X_2}$ for $X_1 \subseteq X_2$ 
(cf.~Remark~\ref{R:subquotient}\eqref{I:subquotient pushing subgroups})
form an inverse system.
\item\label{I:inverse limit}
If $\Intersection_{X \in \calX_V} X = 0$,
then the collection of maps $\pi^{V,X^\perp/X}$ induce a bijection 
\[
	\calI_V \to \varprojlim_{X \in \calX_V} \calI_{X^\perp/X}.
\]
Equip $\calI_V$ with the inverse limit topology.
\item\label{I:measure}
In the remaining parts of this theorem, assume
that $p$ is a prime such that $pV=0$.
Then there exists a unique probability measure $\mu$ on
the Borel $\sigma$-algebra of $\calI_V$ 
such that for every compact open isotropic subgroup $X$ of $(V,Q)$, 
the push-forward $\pi^{V,X^\perp/X}_* \mu$ 
is the uniform probability measure
on the finite set $\calI_{X^\perp/X}$.
\item\label{I:invariant measure}
The measure $\mu$ is invariant under the orthogonal group $\Aut(V,Q)$.
\item
\label{I:intersection distribution in locally compact}
If $Z$ is distributed according to $\mu$,
then $\Prob(Z \textup{ is discrete}) = 1$
and $\Prob(Z \intersect W \textup{ is finite}) = 1$.
If moreover $\dim_{\F_p} V$ is infinite,
then the distribution of $\dim(Z \intersect W)$ is given by $X_{\Sel_p}$ 
(see Definition~\ref{D:X_Sel}).
\end{enumerate}
\end{theorem}

\begin{proof}\hfill
\begin{enumerate}[\upshape (a)]
\item 
The intersection of two compact open isotropic subgroups of $V$
is another one, so $\calX_V$ is a directed poset.
To prove that $\calX_V$ is countable, first consider the bijection
\begin{align}
\label{E:compact opens in W}
	\{\textup{compact open subgroups of $W$}\} 
	&\to \{\textup{finite subgroups of $V/W$}\} \\
\nonumber
	X &\mapsto X^\perp/W.
\end{align}
Since $V/W$ is a countable discrete group,
both sets above are countable.
The map
\begin{align*}
	\calX_V
	&\to \{\textup{compact open subgroups of $W$}\} \\
	X &\mapsto X \intersect W
\end{align*}
has finite fibers, since the $X \in \calX_V$
containing a given compact open subgroup $Y$ of $W$
are in bijection with the isotropic subgroups of the finite group $Y^\perp/Y$.
Thus $\calX_V$ is countable.
\item 
Given $X_1 \subseteq X_2 \subseteq X_3$
the maps $\pi^{X_i^\perp/X_i,X_j^\perp/X_j}$ for $i<j$ 
behave as expected under composition.
\item
The same computation proving \eqref{I:inverse system} shows that the map
is well-defined.
If $X \in \calX_V$, then $X^\perp$ is another compact open subgroup of $V$
since it contains $X$ as a finite-index open subgroup.
The group $(X^\perp)^* \isom V/X$ is a discrete $\F_p$-vector space,
so it equals the direct limit of its finite-dimensional subspaces.
Taking duals shows that $X^\perp$ is the inverse limit of
its finite quotients, i.e., of the groups $X^\perp/Y$ 
where $Y$ ranges over open subgroups of $X^\perp$.
Moreover, every open subgroup of $X^\perp$ contains an open subgroup
of $X$ (just intersect with $X$),
so it suffices to take the latter.

Now the inverse map $(Z_X) \mapsto Z$ is constructed as follows: given 
$(Z_X) \in \varprojlim_{X \in \calX_V} \calI_{X^\perp/X}$,
let
\begin{align*}
	\widetilde{Z}_X &:= \varprojlim_{\substack{Y \in \calX_V \\ Y \subseteq X}} 
			\left( Z_Y \intersect \frac{X^\perp}{Y} \right)
		\subseteq \varprojlim_{\substack{Y \in \calX_V \\ Y \subseteq X}} 
			\frac{X^\perp}{Y}
		= X^\perp \\
	Z &:= \Union_{X \in \calX_V} \widetilde{Z}_X.
\end{align*}
The maps in the inverse system defining $\widetilde{Z}$ are surjections,
so the image of $\widetilde{Z}_X$ in $X^\perp/X$ equals $Z_X$.
If $X, X' \in \calX_V$ and $X' \subseteq X$,
then $\widetilde{Z}_X = \widetilde{Z}_{X'} \intersect X^\perp$,
so $Z \intersect X^\perp = \widetilde{Z}_X$.
Since each $Z_Y$ is isotropic in $Y^\perp/Y$,
the group $\widetilde{Z}_X$ is isotropic,
so $Z$ is isotropic.
If $z \in Z^\perp$, then we have $z \in X^\perp$ for some $X$,
and then for any $Y \subseteq X$,
the element $z \bmod Y \in Y^\perp/Y$
is perpendicular to $\pi^{V,Y^\perp/Y}(Z) = Z_Y$,
but $Z_Y^\perp=Z_Y$, so $z \bmod Y \in Z_Y$,
and also $z \bmod Y \in X^\perp/Y$;
this holds for all $Y \subseteq X$, so $z \in Z$.
Thus $Z^\perp=Z$; i.e., $Z \in \calI_V$.

Now we show that the two constructions are inverse to each other.
If we start with $(Z_X)$, then the $Z$ produced by the inverse map
satisfies $\pi^{V,X^\perp/X}(Z)=Z_X$.
Conversely, if we start with $Z$, and define $Z_X:=\pi^{V,X^\perp/X}(Z)$,
then the inverse map applied to $(Z_X)$ produces $Z'$
such that $Z \intersect X^\perp \subseteq Z'$ for all $X$,
so $Z \subseteq Z'$,
but $Z$ and $Z'$ are both maximal isotropic, so $Z=Z'$.
\item
Since $V/W$ is a discrete $\F_p$-vector space of dimension $\aleph_0$,
we may choose a cofinal increasing sequence 
of finite-dimensional subspaces of $V/W$,
and this corresponds under \eqref{E:compact opens in W} 
to a cofinal decreasing sequence $Y_1,Y_2,\ldots$
of compact open subgroups of $W$ whose intersection is $0$.
Thus \eqref{I:inverse limit} applies.
Each map in the inverse system has fibers of constant size,
by Proposition~\ref{P:counting}\eqref{I:equal size fibers},
so the uniform measures on these finite sets are compatible.
By \cite{BourbakiIntegration1-6}*{III.\S4.5, Proposition~8(iv)},
the inverse limit measure exists.
\item
The construction is functorial 
with respect to isomorphisms $(V,Q) \to (V',Q')$.
\item
Since $\sum_{r=0}^\infty a_r=1$, it suffices to prove the last statement,
that $\Prob(\dim(Z \intersect W)=r) = a_r$.
Let $Y_i$ be as in the proof of~\eqref{I:measure}.
Then $\dim(Z \intersect W)$ 
is the limit of the increasing sequence of nonnegative integers
$\dim(\pi^{V,Y_i^\perp/Y_i}(Z) \intersect \pi^{V,Y_i^\perp/Y_i}(W))$.
By
Proposition\ref{P:counting}\eqref{I:probability of given intersection dim}
and its proof,
the difference of consecutive integers in this sequence
is a sum of independent Bernoulli random variables.
Since $\sum_{j \ge 1} \Prob(B_j=1)$ converges,
the Borel-Cantelli lemma implies that
\[
	\Prob \left(\dim(Z \intersect W) \ne 
	\dim(\pi^{V,Y_i^\perp/Y_i}(Z) \intersect \pi^{V,Y_i^\perp/Y_i}(W)) \right) 
	\to 0
\]
as $i \to \infty$.
In particular, $\Prob\left(\dim(Z \intersect W)=\infty\right)$ is $0$.
On the other hand, $\dim Y_i^\perp/Y_i \to \infty$ as $i \to \infty$, so 
\[
	\Prob\left(\dim(Z \intersect W)=d \right) 
	= \lim_{n \to \infty} a_{d,n} = a_d 
	= \Prob\left(X_{\Sel_p}=d\right). \qedhere
\]
\end{enumerate}
\end{proof}

\begin{remark}
There is only one infinite-dimensional second-countable weakly metabolic
locally compact quadratic $\F_p$-vector space $(V,Q)$,
up to isomorphism.
Inside $V$ we are given a compact open maximal isotropic subspace $W$,
and Theorem~\ref{T:probability measure for LCA group}\eqref{I:intersection distribution in locally compact} implies the existence of a 
discrete maximal isotropic closed subspace $Z$ with $Z \intersect W = 0$.
Since $V$ is infinite and second-countable, $\dim_{\F_p} Z = \aleph_0$,
so the isomorphism type of $Z$ as locally compact abelian group is determined.
The pairing $Z \times W \to \R/\Z$ defined by $(z,w) \mapsto Q(z+w)$
puts $Z$ and $W$ in Pontryagin duality.
Now the summing map $Z \times W \to V$ is an isomorphism
under which $Q$ corresponds to the standard quadratic form
on $Z \times Z^*$ defined in 
Example~\ref{E:hyperbolic locally compact quadratic module}.
\end{remark}

\begin{remark}
Suppose that $V$ is as in 
Theorem~\ref{T:probability measure for LCA group}\eqref{I:intersection distribution in locally compact},
and that $Z \in \calI_V$ is chosen at random.
The probability that $Z$ contains a given nonzero vector $v$ of $V$
then equals $0$, because for any compact open isotropic subgroup $X \le V$
small enough that $v \in X^\perp - X$, if $\dim X^\perp/X=2n$,
then the probability that 
$\pi^{V,X^\perp/X}(Z)$ contains the nonzero element $\pi^{V,X^\perp/X}(v)$
is $1/(p^n+1)$, which tends to $0$ as $X$ shrinks.
Now, if we fix a {\em discrete} $Z_0 \in \calI_V$
and choose $Z \in \calI_V$ at random,
then $\dim(Z \intersect Z_0)=0$ with probability $1$
by the previous sentence applied to each nonzero vector
of the countable set $Z_0$.
\end{remark}

\subsection{Moments}

Given a random variable $X$, let $\Exp(X)$ be its expectation.
So if $m \in \Z_{\ge 0}$, then $\Exp(X^m)$ is its $m^\tH$ moment.

\begin{proposition}
\label{P:moments}
Fix a prime $p$ and fix $m \in \Z_{\ge 0}$.
Let $X_n$ be as in 
Proposition~\ref{P:counting}\eqref{I:probability of given intersection dim},
and let $X_{\Sel_p}$ be as in Definition~\ref{D:X_Sel}.
Then
\begin{enumerate}[\upshape (a)]
\item\label{I:moments}
We have
\begin{align*}
  \Exp\left(\left(p^{X_n}\right)^m\right) 
	&= \prod_{i=1}^m \frac{p^i+1}{1+p^{-(n-i)}} \\
  \Exp\left(\left(p^{X_{\Sel_p}}\right)^m\right) 
	&= \prod_{i=1}^m (p^i+1).
\end{align*}
In particular, $\Exp(p^{X_{\Sel_p}}) = p+1$.
\item 
We have $\Prob(X_n \textup{ is even})=1/2$ for each $n>0$,
and $\Prob(X_{\Sel_p} \textup{ is even})=1/2$.
\item 
If we condition on the event that $X_{\Sel_p}$ has a prescribed parity,
the moments in~\eqref{I:moments} remain the same.
The same holds for the $m^{\textup{th}}$ moment of $p^{X_n}$ if $m<n$.
\end{enumerate}
\end{proposition}

\begin{proof}\hfill
  \begin{enumerate}[\upshape (a)]
  \item 
Substitute $z=p^m$ in 
Proposition~\ref{P:counting}\eqref{I:generating function}.
The products telescope.
  \item
Substitute $z=-1$ in 
Proposition~\ref{P:counting}\eqref{I:generating function}.
\item 
Substitute $z=-p^m$ in 
Proposition~\ref{P:counting}\eqref{I:generating function}.
  \end{enumerate}
\end{proof}

\section{Shafarevich-Tate groups of finite group schemes}
\label{S:Sha}

\subsection{Definitions}
\label{S:definition of Sha}

For each field $k$, choose an algebraic closure $\kbar$
and a separable closure $k_s \subseteq \kbar$,
and let $G_k:=\Gal(k_s/k)$.

A \defi{local field} is a nondiscrete locally compact topological field;
each such field is a finite extension of one of $\R$, $\Q_p$, or $\F_p((t))$
for some prime $p$.
A \defi{global field} is a finite extension of $\Q$ or $\F_p(t)$
for some prime $p$.
Let $\Omega$ be the set of nontrivial places of $k$.
For $v \in \Omega$, let $k_v$ be the completion of $k$ at $v$,
so $k_v$ is a local field; if $v$ is nonarchimedean, let $\calO_v$
be the valuation ring in $k_v$.

For a sheaf of abelian groups $M$ on the big fppf site of $\Spec k$, 
define
\begin{equation}
\label{E:Sha}
	\Sha^1(k,M) := 
	\ker\left( \HH^1(k,M) \to \prod_{v \in \Omega} \HH^1(k_v,M) \right).
\end{equation}

\begin{remark}
If $M$ is represented by a smooth finite-type group scheme over $k$,
such as the kernel of an isogeny of degree not divisible by $\Char k$,
then we may interpret the cohomology groups as Galois cohomology groups:
$\HH^1(k,M) = \HH^1(G_k,M(k_s))$ and so on.
\end{remark}

\begin{definition}
\label{D:theta characteristic torsor}
Let $X$ be a smooth projective curve over $k$.
Let $A=\Jac X$.
The degree $g-1$ component of the Picard scheme of $X$
contains a closed subscheme $\calT$ parametrizing line sheaves on $X$
whose square is isomorphic to the canonical sheaf of $X$.
This $\calT$ is a torsor under $A[2]$, 
called the \defi{theta characteristic torsor}.
Let $c_\calT \in \HH^1(k,A[2])$ be its class.
\end{definition}

\subsection{Vanishing criteria}

The following criteria for vanishing of $\Sha^1(k,M)$ 
will be especially relevant for
Theorem~\ref{T:Selmer is intersection}\eqref{I:quotient of Sel}.

\begin{proposition}
\label{P:Sha^1=0}
Suppose that $M$ is a finite \'etale group scheme over $k$,
so we identify $M$ with the finite $G_k$-module $M(k_s)$.
\begin{enumerate}[\upshape (a)]
\item\label{I:Sha of Z/nZ}
If $M=\Z/n\Z$, then $\Sha^1(k,M)=0$.
\item\label{I:Sha of direct summand} 
If $M$ is a direct summand of a direct sum of 
permutation $\Z/n\Z$-modules arising from finite separable extensions of $k$,
then $\Sha^1(k,M)=0$.
\item\label{I:Sha is in intersection}
Let $G$ be the image of $G_k$ in $\Aut M(k_s)$.  
Identify $\HH^1(G,M)$ with its image under the injection
$\HH^1(G,M) \injects \HH^1(k,M)$.
Then
\[
	\Sha^1(k,M) \subseteq
	\Intersection_{\textup{cyclic $H \le G$}} 
	\ker\left( \HH^1(G,M)\to \HH^1(H,M) \right).
\]
\item\label{I:Sha of cyclic Sylow} 
If $p$ is a prime such that $pM=0$ 
and the Sylow $p$-subgroups of $\Aut M(k_s)$ are cyclic, then $\Sha^1(k,M)=0$.
\item\label{I:Sha of E[p]}
 If $E$ is an elliptic curve, and $p \ne \Char k$,
then $\Sha^1(k,E[p])=0$.
\item\label{I:Sha of A[2]} 
If $\Char k \ne 2$, and $A$ is the Jacobian of the smooth projective
model of $y^2=f(x)$, where $f \in k[x]$ is separable of odd degree,
then $\Sha^1(k,A[2])=0$.
\end{enumerate}
\end{proposition}

\begin{proof}\hfill
  \begin{enumerate}[\upshape (a)]
  \item See \cite{MilneADT2006}*{Example~I.4.11(i)}.
  \item Combine \eqref{I:Sha of Z/nZ} with 
Shapiro's lemma \cite{Atiyah-Wall1967}*{\S4,~Proposition~2}
to obtain
the result for a finite permutation module $(\Z/n\Z)[G_k/G_L]$
for a finite separable extension $L$ of $k$.
The result for direct summands of direct sums of these follows.
\item This is a consequence of the Chebotarev density theorem:
see \cite{Bruin-Poonen-Stoll-preprint}.
\item Let $G$ be the image of $G_k \to \Aut M(k_s)$.
Then any Sylow $p$-subgroup $P$ of $G$ is cyclic.
But the restriction $\HH^1(G,M) \to \HH^1(P,M)$ 
is injective \cite{Atiyah-Wall1967}*{\S6,~Corollary~3},
so \eqref{I:Sha is in intersection} shows that $\Sha^1(k,M)=0$.
\item Any Sylow-$p$-subgroup of $\GL_2(\F_p)$ is conjugate to the group
of upper triangular unipotent matrices, which is cyclic.
Apply~\eqref{I:Sha of cyclic Sylow}.
\item 
The group $A[2]$ is a direct summand of the permutation $\Z/2\Z$-module
on the set of zeros of $f$.
Apply~\eqref{I:Sha of direct summand}.\qedhere
\end{enumerate}
\end{proof}

\subsection{Jacobians of hyperelliptic curves}
\label{S:Jacobians of hyperelliptic curves}

See \cite{Poonen-Rains-selfcup-preprint}*{Example~3.12(b)} 
for a $2$-dimensional Jacobian $A$ with $\Sha^1(\Q,A[2]) \ne 0$.
Such examples are rare:
a special case of Proposition~\ref{P:Sha usually 0} below
shows that asymptotically 100\% of $2$-dimensional Jacobians $A$ over $\Q$
have $\Sha^1(\Q,A[2]) = 0$.

\begin{proposition}\label{P:Sha usually 0}
Fix $g \ge 1$ and a prime $p$.
For random $f(x) \in \Z[x]$ of degree $2g+1$ 
with coefficients in $[-B,B]$,
if $A$ is the Jacobian of the smooth projective model $X$ of $y^2=f(x)$,
the probability that $\Sha^1(\Q,A[p])=0$ tends to $1$ as $B \to \infty$.
The same holds if $2g+1$ is replaced by $2g+2$ 
(the general case for a genus~$g$ hyperelliptic curve).
\end{proposition}

\begin{proof}
By Proposition~\ref{P:Sha^1=0}\eqref{I:Sha of E[p]} we may assume $g \ge 2$.
We may assume that $f$ is separable.

First consider the case $p \ne 2$.
Generically, the image of $G_\Q \to \Aut A[p]$ is as large as possible
given the existence of the Weil pairing $e_p$,
i.e., isomorphic to $\GSp_{2g}(\F_p)$.
By the Hilbert irreducibility theorem, 
the same holds for asymptotically $100\%$ of the polynomials $f$.
By Proposition~\ref{P:Sha^1=0}\eqref{I:Sha is in intersection}, 
$\Sha^1(\Q,A[p])$ is contained in the subgroup
$\HH^1(\GSp_{2g}(\F_p),A[p])$ of $\HH^1(\Q,A[p])$,
and that subgroup is $0$ because the central element $-I \in \GSp_{2g}(\F_p)$
has no fixed vector 
(cf.~\cite{Foulser1964}*{Lemma~14.4} and \cite{Pollatsek1971}*{Theorem~2.3}).

Now suppose that $p=2$.
In the degree $2g+1$ case, we are done by 
Proposition~\ref{P:Sha^1=0}\eqref{I:Sha of A[2]}.
So assume that $\deg f=2g+2$.
Let $\Delta$ be the set of zeros of $f$ in $\Qbar$, so $\#\Delta=2g+2$.
For $m \in \Z/2\Z$, let $\calW_m$ be the quotient of the sum-$m$ part of
the permutation module $\F_2^\Delta\isom \F_2^{2g+2}$ 
by the diagonal addition action of $\F_2$.
Then the $G_\Q$-module $A[2]$ may be identified with $\calW_0$,
and $\calW_m$ is a torsor under $\calW_0$.

Again by the Hilbert irreducibility theorem, we may assume that
the group $\Gal(\Q(A[2])/\Q) \isom \Gal(f)$ is as large as possible, 
i.e., equal to $S_{2g+2}$.
Then
$\Sha^1(\Q,A[2]) \subseteq \HH^1(S_{2g+2},\calW_0) \subset \HH^1(\Q,A[2])$.
The group $\HH^1(S_{2g+2},\calW_0)$ is of order~$2$,
generated by the class $c_{\calW_1}$ 
of $\calW_1$ \cite{Pollatsek1971}*{Theorem~5.2}.
Computations as in \cite{Poonen-Stoll1999}*{\S9.2}
show that for each prime $\ell$, the probability that $f(x)$
factors over $\Z_\ell$ into irreducible polynomials of degree $2g$ and $2$
defining unramified and ramified extensions of $\Q_\ell$, respectively,
is of order $1/\ell$ (not smaller) as $\ell \to \infty$, and in this case 
no point in $\calW_1$ is $G_{\Q_\ell}$-invariant,
so $c_{\calW_1}$ has nonzero image in $\HH^1(\Q_\ell,A[2])$.
Since the conditions at finitely many $\ell$ are asymptotically independent
as $B \to \infty$, and since $\sum 1/\ell$ diverges,
there will exist such a prime $\ell$ for almost all $f$,
and in this case $c_{\calW_1} \notin \Sha^1(\Q,A[2])$,
so $\Sha^1(\Q,A[2])=0$.
\end{proof}

\begin{remark}
\label{R:extension of Sha usually 0}
Proposition~\ref{P:Sha usually 0} can easily be extended 
to an arbitrary global field of characteristic not equal to $2$ or $p$.
\end{remark}

\begin{remark}
\label{R:theta characteristic torsor}
Let $X$ be the smooth projective model of $y^2=f(x)$,
where $f(x) \in k[x]$ is separable of degree $2g+2$.
As torsors under $A[2] \isom \calW_0$, 
we have $\calT \isom \calW_{g-1}$ 
(cf.~\cite{Mumford1971}*{p.~191}).  
\end{remark}

\subsection{Jacobians with generic $2$-torsion}
\label{S:Jacobians with generic 2-torsion}

Suppose that $X$ is a curve of genus $g \ge 2$ 
over a global field $k$ of characteristic not~$2$
such that the image $G$ of $G_k \to \Aut A[2]$ is as large as possible,
i.e., $\Sp_{2g}(\F_2)$.
(This forces $X$ to be non-hyperelliptic if $g \ge 3$.)
By \cite{Pollatsek1971}*{Theorems~4.1~and~4.8},
the group $\HH^1(G,A[2]) \subseteq \HH^1(G_k,A[2])$
is of order $2$, generated by $c_\calT$.
So Proposition~\ref{P:Sha^1=0}\eqref{I:Sha is in intersection}
shows that $\Sha^1(k,A[2])$ is of order $2$ or $1$,
according to whether the nonzero class $c_\calT$
lies in $\Sha^1(k,A[2])$ or not.

\section{Selmer groups as intersections of two maximal isotropic subgroups}
\label{S:Selmer}

\subsection{Quadratic form arising from the Heisenberg group}
\label{S:Heisenberg}

Let $A$ be an abelian variety over a field $k$.
Let $\Adual$ be its dual abelian variety.
Let $\lambda \colon A \to \Adual$ be an isogeny equal to its dual.
The exact sequence
\[
	0 \to A[\lambda] \to A \stackrel{\lambda}\to \Adual \to 0
\]
gives rise to the ``descent sequence''
\begin{equation}
\label{E:descent sequence}
	0 \to 
	\frac{\Adual(k)}{\lambda A(k)} \stackrel{\delta}\to 
	\HH^1(A[\lambda]) \to 
	\HH^1(A)[\lambda] \to 
	0,
\end{equation}
where $\HH^1(A)[\lambda]$ is the kernel of the homomorphism
$\HH^1(\lambda) \colon \HH^1(A) \to \HH^1(\Adual)$.

Since $\lambda$ is self-dual, we obtain an (alternating) Weil pairing
\[
	e_\lambda \colon A[\lambda] \times A[\lambda] \to \G_m
\]
identifying $A[\lambda]$ with its own Cartier dual 
(cf.~\cite{MumfordAV1970}*{p.~143,~Theorem~1}).
Composing the cup product with $\HH^1(e_\lambda)$ gives a symmetric pairing
\[
	\underset{e_\lambda}\cup\colon 
	\HH^1(A[\lambda]) \times \HH^1(A[\lambda]) \to \HH^2(\G_m)
\]
and its values are killed by $\deg \lambda$.
It is well-known (especially when $\lambda$ is separable) 
that the image of the natural map
$\Adual(k)/\lambda A(k) \to \HH^1(A[\lambda])$ 
is isotropic with respect to $\underset{e_\lambda}\cup$.

In the rest of Section~\ref{S:Selmer},
we will assume that we are in one of the following cases:
\begin{enumerate}[\upshape I.]
\item The self-dual isogeny $\lambda$ has odd degree.
\item The self-dual isogeny $\lambda$ is of the form $\phi_\LL$ 
for some symmetric line sheaf $\LL$ on $A$.
(\defi{Symmetric} means $[-1]^*\LL \isom \LL$.  
For the definition of $\phi_\LL$,
see~\cite{MumfordAV1970}*{p.~60 and Corollary~5 on p.~131}.)
\end{enumerate}
In each of these cases, we will construct a natural quadratic form $q$
whose associated bilinear pairing is $\underset{e_\lambda}\cup$.
Moreover, in each case, we will show that the image
of $\Adual(k)/\lambda A(k) \to \HH^1(A[\lambda])$ 
is isotropic with respect to $q$.

\begin{remark}
We will need both cases, since, for example, we want to study $\Sel_p E$
for every prime $p$, including $2$.
\end{remark}

\begin{remark}
For any symmetric line sheaf $\LL$,
the homomorphism $\phi_\LL$ is self-dual \cite{Polishchuk2003}*{p.~116}.
If moreover $\LL$ is ample, then $\phi_\LL$ is an isogeny
(see \cite{MumfordAV1970}*{p.~124 and Corollary~5 on p.~131}).
\end{remark}

\begin{remark}
The obstruction to expressing a self-dual isogeny $\lambda$
as $\phi_\LL$ is given by an element $c_\lambda \in \HH^1(\Adual[2])$.
For example, if $\lambda$ is an odd multiple of the principal polarization
on a Jacobian of a curve with no rational theta characteristic,
then $c_\lambda \ne 0$.
See~\cite{Poonen-Rains-selfcup-preprint}*{\S3} for these facts,
and for many criteria for the vanishing of $c_\lambda$.
\end{remark}

\begin{remark}
\label{R:2mu}
If $\lambda=2\mu$ for some self-dual isogeny $\mu \colon A \to \Adual$,
then $\lambda$ is of the form $\phi_\LL$, 
since $c_\lambda=2c_\mu=0$.
Explicitly, take $\LL:=(1,\mu)^* \Poincare$,
where $\Poincare$ is the Poincar\'e line sheaf on $A \times \Adual$ 
(see \cite{MumfordAV1970}*{\S20, proof of Theorem~2}).
\end{remark}

We now return to the construction of the quadratic form.

\bigskip

{\em Case~I: $\deg \lambda$ is odd.}

Then $q(x):= -\frac12 \left( x \underset{e_\lambda}\cup x \right)$
is a quadratic form
whose associated bilinear pairing is $-\underset{e_\lambda}\cup$.
(The sign here is chosen to make the conclusion of 
Corollary~\ref{C:construction of quadratic form} hold for $q$.)
Since the image of $\Adual(k)/\lambda A(k) \to \HH^1(A[\lambda])$ 
is isotropic with respect to $\underset{e_\lambda}\cup$, 
it is isotropic with respect to $q$ too.

\bigskip

{\em Case~II: there is a symmetric line sheaf $\LL$ on $A$
such that $\lambda=\phi_\LL$.}  (This hypothesis remains in force
until the end of Section~\ref{S:Heisenberg}.)

When $\lambda$ is separable, 
Zarhin~\cite{Zarhin1974}*{\S2} constructed a quadratic form
$q \colon \HH^1(A[\lambda]) \to \HH^2(\G_m)$
whose associated bilinear pairing was $\underset{e_\lambda}\cup$;
for elliptic curves, C.~O'Neil showed that 
the image of $\Adual(k)/\lambda A(k) \to \HH^1(A[\lambda])$ 
is isotropic for $q$
(this is implicit in~\cite{O'Neil2002}*{Proposition~2.3}).
Because we wish to include the inseparable case,
and because we wish to prove isotropy of the quadratic form
for abelian varieties of arbitrary dimension,
we will give a detailed construction and proof in the general case.

The pairs $(x,\phi)$ where $x \in A(k)$
and $\phi$ is an isomorphism from $\LL$ to $\tau_x^* \LL$
form a (usually nonabelian) group under the operation
\[
	(x,\phi)(x',\phi') = (x+x',(\tau_{x'} \phi) \phi').
\]
The same can be done after base extension,
so we get a group functor.
Automorphisms of $\LL$ induce the identity on this group functor,
so it depends only on the class of $\LL$ in $\Pic A$.

\begin{proposition}[Mumford]
\label{P:Heisenberg}
\hfill
\begin{enumerate}[\upshape (a)]
\item 
This functor is representable by a finite-type group scheme $\calH(\LL)$,
called the \defi{Heisenberg group} 
(or \defi{theta group} or \defi{Mumford group}).
\item
It fits in an exact sequence
\begin{equation}
\label{E:Heisenberg extension}
	1 \to \G_m \to \calH(\LL) \to A[\lambda] \to 1,
\end{equation}
where the two maps in the middle are given by 
$t \mapsto (0,\textup{multiplication by $t$})$
and $(x,\phi)$ to $x$.
This exhibits $\calH(\LL)$ as a central extension of finite-type group schemes.
\item
The induced commutator pairing
\[
	A[\lambda] \times A[\lambda] \to \G_m
\]
is the Weil pairing $e_\lambda$.
\end{enumerate}
\end{proposition}

\begin{proof}
See \cite{MumfordTheta3}*{pp.~44--46}.
\end{proof}

\begin{corollary}
\label{C:construction of quadratic form}
The connecting homomorphism $q \colon \HH^1(A[\lambda]) \to \HH^2(\G_m)$
induced by \eqref{E:Heisenberg extension} 
is a quadratic form whose associated bilinear pairing
$\HH^1(A[\lambda]) \times \HH^1(A[\lambda]) \to \HH^2(\G_m)$
sends $(x,y)$ to $-x \underset{e_\lambda}\cup y$.
\end{corollary}

\begin{proof}
Applying \cite{Poonen-Rains-selfcup-preprint}*{Proposition~2.9} 
to \eqref{E:Heisenberg extension}
shows that $q$ is a quadratic {\em map}
giving rise to the bilinear pairing claimed.
By Remark~\ref{R:quadratic map vs quadratic form},
it remains to prove the identity $q(-v)=q(v)$.
Functoriality of \eqref{E:Heisenberg extension} 
with respect to the automorphism $[-1]$ of $A$ gives a commutative diagram
\[
\xymatrix{
1 \ar[r] & \G_m \ar[r] \ar@{=}[d] & \calH(\LL) \ar[r] \ar[d] & A[\lambda] \ar[r] \ar[d]^{-1} & 1 \\
1 \ar[r] & \G_m \ar[r] & \calH([-1]^* \LL) \ar[r] & A[\lambda] \ar[r] & 1 \\
}
\]
But $[-1]^* \LL \isom \LL$, so both rows give rise to $q$.
Functoriality of the connecting homomorphism gives $q(-v)=q(v)$
for any $v \in \HH^1(A[\lambda])$.
\end{proof}

\begin{remark}
The proof of the next proposition
involves a sheaf of sets on the big fppf site of $\Spec k$,
but in the case $\Char k \nmid \deg \lambda$,
it is sufficient to think of it as a set with a continuous $G_k$-action.
\end{remark}

\begin{proposition}
\label{P:rational points give isotropic subgroup}
Identify $\Adual(k)/\lambda A(k)$
with its image $W$ under $\delta$ in \eqref{E:descent sequence}.
Then $q|_W=0$.
\end{proposition}

\begin{proof}
Let $\Poincare$ be the Poincar\'e line sheaf on $A \times \Adual$.
For $y \in \Adual(k)$, let $\Poincare_y$ be the line sheaf on $A$
obtained by restricting $\Poincare$ to $A \times \{y\}$.
For any $y_1,y_2 \in \Adual(k)$, there is a canonical isomorphism
$\iota_{y_1,y_2} \colon \Poincare_{y_1} \tensor \Poincare_{y_2} 
\to \Poincare_{y_1+y_2}$, 
satisfying a cocycle condition \cite{Polishchuk2003}*{\S10.3}.

The group $\calH(\LL)(k)$ acts on the left
on the set of triples $(x,y,\phi)$ where $x \in A(k)$, $y \in \Adual(k)$,
and $\phi \colon \LL \tensor \Poincare_y \to (\tau_x^* \LL)$
as follows:
\[
	(x,\phi)(x',y',\phi') = (x+x',y',(\tau_{x'} \phi) \phi').
\]
The same holds after base extension,
and we get an fppf-sheaf of sets $\calG(\LL)$ 
on which $\calH(\LL)$ acts freely.
There is a morphism $\calG(\LL) \to \Adual$ 
sending $(x,y,\phi)$ to $y$,
and this identifies $\Adual$ 
with the quotient sheaf $\calH(\LL) \backslash \calG(\LL)$.
There is also a morphism $\calG(\LL) \to A$
sending $(x,y,\phi)$ to $x$,
and the action of $\calH(\LL)$ on $\calG(\LL)$
is compatible with the action of its quotient $A[\lambda]$ on $A$.
Thus we have the following compatible diagram
\[
\xymatrix{
1 \ar[r] & \calH(\LL) \ar@{..>}[r] \ar[d] & \calG(\LL) \ar[r] \ar[d] & \Adual \ar[r] \ar[d] & 0 \\
0 \ar[r] & A[\lambda] \ar[r] & A \ar[r]^\lambda & \Adual \ar[r] & 0 \\
}
\]
where the first row indicates only that $\calH(\LL)$ acts
freely on $\calG(\LL)$ with quotient being $\Adual$.
This is enough to give a commutative square of pointed sets
\[
\xymatrix{
\HH^0(\Adual) \ar[r] \ar@{=}[d] & \HH^1(\calH(\LL)) \ar[d] \\
\HH^0(\Adual) \ar[r] & \HH^1(A[\lambda]), \\
}
\]
so $\HH^0(\Adual) \to \HH^1(A[\lambda])$ factors through $\HH^1(\calH(\LL))$.
But the sequence 
$\HH^1(\calH(\LL)) \to \HH^1(A[\lambda]) \to \HH^2(\G_m)$ 
from \eqref{E:Heisenberg extension}
is exact,
so the composition $\HH^0(\Adual) \to \HH^1(A[\lambda]) \to \HH^2(\G_m)$
is $0$.
\end{proof}

\begin{remark}
Proposition~\ref{P:rational points give isotropic subgroup}
can be generalized to an abelian scheme 
over an arbitrary base scheme $S$.
The proof is the same.
\end{remark}

\subsection{Local fields}
\label{S:local fields}

Let $k_v$ be a local field.
The group $\HH^1(k_v,A[\lambda])$ has a topology making it locally compact,
the group $\HH^2(k_v,\G_m)$ may be identified with a subgroup of $\Q/\Z$
and given the discrete topology,
and the quadratic form $q \colon \HH^1(k_v,A[\lambda]) \to \HH^2(k_v,\G_m)$
is continuous (cf.~\cite{MilneADT2006}*{III.6.5}; the same arguments work
even though \eqref{E:Heisenberg extension} has a nonabelian group
in the middle).
The composition
\[
	\HH^1(k_v,A[\lambda]) \stackrel{q}\to \HH^2(k_v,\G_m) \injects \Q/\Z.
\]
is a quadratic form $q_v$.
By local duality~\cite{MilneADT2006}*{I.2.3,~I.2.13(a),~III.6.10}, 
$q_v$ is nondegenerate.
Moreover, $\HH^1(k_v,A[\lambda])$ is 
finite if $\Char k_v \nmid \deg \lambda$ \cite{MilneADT2006}*{I.2.3,~I.2.13(a)}, and 
$\sigma$-compact in general \cite{MilneADT2006}*{III.6.5(a)}, 
so it is second-countable by Corollary~\ref{C:second-countable}.

\begin{proposition}
\label{P:local points are maximal isotropic}
Let $k_v$ be a local field.
In~\eqref{E:descent sequence} for $k_v$,
the group $W \isom \Adual(k_v)/\lambda A(k_v)$ 
is a compact open maximal isotropic subgroup of $(\HH^1(k_v,A[\lambda]),q_v)$,
which is therefore weakly metabolic.
\end{proposition}

\begin{proof}
By Proposition~\ref{P:rational points give isotropic subgroup},
$q_v$ restricts to $0$ on $W$,
so it suffices to show that $W^\perp=W$.
Let $A(k_v)_\bullet$ be $A(k_v)$ modulo its connected component
(which is nonzero only if $k_v$ is $\R$ or $\C$).
Then $W$ is the image of $\Adual(k_v)_\bullet \to \HH^1(k_v,A[\lambda])$,
so $W^\perp$ is the kernel of the dual map,
which by Tate local duality \cite{MilneADT2006}*{I.3.4,~I.3.7,~III.7.8}
is $\HH^1(k_v,A[\lambda]) \to \HH^1(k_v,A)$.
This kernel is $W$, by exactness of~\eqref{E:descent sequence}.
\end{proof}

Suppose that $k_v$ is nonarchimedean.
Let $\calO_v$ is its valuation ring and let $\F_v$ be its residue field.
Suppose that $A$ has \defi{good reduction},
i.e., that it extends to an abelian scheme (again denoted $A$) over $\calO_v$.
Then the fppf-cohomology group $\HH^1(\calO_v,A[\lambda])$ 
is an open subgroup of $\HH^1(k_v,A[\lambda])$.

\begin{remark}
If moreover $\Char \F_v \nmid \deg \lambda$,
then we may understand $\HH^1(\calO_v,A[\lambda])$ in concrete terms as 
the kernel $\HH^1(k_v,A[\lambda])_\unr$ of the restriction map 
\[
	\HH^1(k_v,A[\lambda]) \to \HH^1(k_v^\unr,A[\lambda])
\]
of Galois cohomology groups.
\end{remark}

\begin{proposition}
\label{P:unramified is isotropic}
Suppose that $k_v$ is nonarchimedean.
Suppose that $A$ extends to an abelian scheme over $\calO_v$.
Then the subgroups $W \isom \Adual(k_v)/\lambda A(k_v)$
and $\HH^1(\calO_v,A[\lambda])$ in $\HH^1(k_v,A[\lambda])$ are equal.
In particular, $\HH^1(\calO_v,A[\lambda])$ is a maximal isotropic subgroup.
\end{proposition}

\begin{proof}
By~\cite{MilneEtaleCohomology1980}*{III.3.11(a)} and~\cite{Lang1956},
respectively, $\HH^1(\calO_v,A) \isom \HH^1(\F_v,A) = 0$.
The valuative criterion for properness~\cite{Hartshorne1977}*{II.4.7}
yields $A(\calO_v) = A(k_v)$ and $\Adual(\calO_v)=\Adual(k_v)$.
So taking cohomology of \eqref{E:descent sequence} over $\calO_v$
gives the result.
\end{proof}

\subsection{Global fields}
\label{S:global fields}

Let $k$ be a global field.
For any nonempty subset $\calS$ of $\Omega$ containing the archimedean places,
define the \defi{ring of $\calS$-integers} 
$\calO_\calS:=\{x \in k : v(x) \ge 0 \textup{ for all $v \notin \calS$}\}$.

Let $A$ be an abelian variety over $k$,
Let $\lambda \colon A \to \Adual$ be a self-dual isogeny 
as in Case I or~II of Section~\ref{S:Heisenberg}.
Choose a nonempty finite $\calS$ \defi{containing all bad places},
by which we mean that $\calS$ contains all archimedean places 
and $A$ extends to an abelian scheme $A$ over $\calO_\calS$.
In Example~\ref{E:restricted direct product}
take $I=\Omega$, $V_i=\HH^1(k_v,A[\lambda])$,
$Q_i=q_v$, and $W_i=\Adual(k_v)/\lambda A(k_v)$,
which is valid by Proposition~\ref{P:local points are maximal isotropic}.
The resulting restricted direct product
\[
	V:=\sideset{}{'}\prod_{v \in \Omega} 
	\left(\HH^1(k_v,A[\lambda]),\frac{\Adual(k_v)}{\lambda A(k_v)} \right)
\]
equipped with the quadratic form
\begin{align*}
	Q \colon
	\sideset{}{'}\prod_{v \in \Omega} 
	\left(\HH^1(k_v,A[\lambda]),\frac{\Adual(k_v)}{\lambda A(k_v)} \right)
	&\to \Q/\Z \\
	(\xi_v)_{v \in \Omega} &\mapsto \sum_v q_v(\xi_v).
\end{align*}
is a second-countable weakly metabolic locally compact quadratic module.
Proposition~\ref{P:unramified is isotropic},
which applies for all but finitely many $v$,
shows that 
\[
	V = \sideset{}{'}\prod_{v \in \Omega} 
	\left(\HH^1(k_v,A[\lambda]),\HH^1(\calO_v,A[\lambda]) \right).
\]
(The subgroup $\HH^1(\calO_v,A[\lambda])$ is defined and equal to
$\Adual(k_v)/\lambda A(k_v)$ only for $v \notin \calS$,
but that is enough.)

As usual, define the \defi{Selmer group}
\[
	\Sel_\lambda A := \ker\left( \HH^1(k,A[\lambda]) 
		\to \prod_{v \in \Omega} \HH^1(k_v,A) \right).
\]

Below will appear $\Sha^1(k,A[\lambda])$,
which is a subgroup of $\Sel_\lambda A$,
and is not to be confused with 
the \defi{Shafarevich-Tate group} $\Sha(A)=\Sha^1(k,A)$.

\begin{theorem}
\label{T:Selmer is intersection}
\hfill
\begin{enumerate}[\upshape (a)]
\item
The images of the homomorphisms
\[
\xymatrix{
& \HH^1(k,A[\lambda]) \ar[d] \\
\displaystyle \prod_{v \in \Omega} \dfrac{\Adual(k_v)}{\lambda A(k_v)} \ar[r] 
	& \displaystyle 
	\sideset{}{'}\prod_{v \in \Omega} 
	(\HH^1(k_v,A[\lambda]),\HH^1(\calO_v,A[\lambda]))
}
\]
are maximal isotropic subgroups with respect to $Q$.
\item\label{I:quotient of Sel}
The vertical map induces an isomorphism from
$\Sel_\lambda A/\Sha^1(k,A[\lambda])$
to the intersection of these two images.
(See Section~\ref{S:Sha} for information about $\Sha^1(k,A[\lambda])$,
which is often~$0$.)
\end{enumerate}
\end{theorem}

\begin{proof}\hfill
  \begin{enumerate}[\upshape (a)]
  \item 
The subgroup 
$\prod_{v \in \Omega} \Adual(k_v)/\lambda A(k_v)$ 
(or rather its image $W$ under the horizontal injection)
is maximal isotropic by construction.

The vertical homomorphism 
$\HH^1(k,A[\lambda]) \to 
	\sideset{}{'}\prod_{v \in \Omega} 
	(\HH^1(k_v,A[\lambda]),\HH^1(\calO_v,A[\lambda]))$
is well-defined since 
each element of $\HH^1(k,A[\lambda])$ belongs to the subgroup
$\HH^1(\calO_\calT,A[\lambda])$ for some finite $\calT \subseteq \Omega$
containing $\calS$,
and $\calO_{\calT} \subseteq \calO_v$ for all $v \notin \calT$.
Let $W$ be the image.
Suppose that $s \in \HH^1(k,A[\lambda])$,
and let $w \in W$ be its image.
The construction of the quadratic form 
of Corollary~\ref{C:construction of quadratic form}
is functorial with respect to base extension,
so $Q(w)$ can be computed by evaluating the global quadratic form 
\[
	q \colon \HH^1(k,A[\lambda]) \to \HH^2(k,\G_m)
\]
on $s$, and afterwards summing the local invariants.
Exactness of
\[
	0 \to \HH^2(k,\G_m) \to \Directsum_{v \in \Omega} \HH^2(k_v,\G_m) \stackrel{\sum \inv}\longrightarrow \Q/\Z \to 0
\]
in the middle 
(the reciprocity law for the Brauer group: 
see~\cite{Gille-Szamuely2006}*{Remark~6.5.6} for references)
implies that the sum of the local invariants of our global class is $0$.
Thus $Q|_W=0$.

It remains to show that $W$ is its own annihilator.
Since $e_\lambda$ identifies $A[\lambda]$ with its own Cartier dual,
the middle three terms of the $9$-term Poitou-Tate exact sequence 
(\cite{MilneADT2006}*{I.4.10(c)} and \cite{GonzalezAviles2009}*{4.11})
give the self-dual exact sequence
\[
		\HH^1(k,A[\lambda]) 
	\stackrel{\beta_1}\longrightarrow 
		\sideset{}{'}\prod_{v \in \Omega} 
		(\HH^1(k_v,A[\lambda]),\HH^1(\calO_v,A[\lambda]))
	\stackrel{\gamma_1}\longrightarrow 
		\HH^1(k,A[\lambda])^*,
\]
where $*$ denotes Pontryagin dual.
Since $W=\im(\beta_1)$ and the dual of $\beta_1$ is $\gamma_1$,
\[
	W^\perp = \ker(\gamma_1) = \im(\beta_1) = W.
\]
\item
This follows from the exactness of~\eqref{E:descent sequence}
for $S=\Spec k_v$ for each $v \in \Omega$. \qedhere
  \end{enumerate}
\end{proof}

\begin{remark}
\label{R:variant}
There is a variant of 
Theorem~\ref{T:Selmer is intersection}
in which the infinite restricted direct product is taken
over only a {\em subset} $\calS$ of $\Omega$
containing all bad places 
and all places of residue characteristic dividing $\deg \lambda$.
If $\calS$ is finite, then the restricted direct product
becomes a finite direct product.
The same proof as before shows that 
the images of $\prod_{v\in \calS} \Adual(k_v)/\lambda A(k_v)$ 
and $\HH^1(\calO_\calS,A[\lambda])$ are maximal isotropic.
The intersection of the images equals the image of $\Sel_\lambda A$.
\end{remark}

\begin{remark}
\label{R:4-torsion again}
Suppose that $A = \Jac X$ and $\lambda$ is multiplication-by-$2$,
with $\LL$ as in Remark~\ref{R:2mu}.
Let $c_\calT$ be as in Definition~\ref{D:theta characteristic torsor},
and let $c_{\calT,v}$ be its image in $\HH^1(k_v,A[2])$.
It follows from \cite{Poonen-Stoll1999}*{Corollary~2}
that $c_\calT \in \Sel_2 A$.
By \cite{Poonen-Rains-selfcup-preprint}*{Theorem~3.9}, 
\begin{equation}
  \label{E:self cup product}
	x \underset{e_2}{\cup} x = x \underset{e_2}{\cup} c_\calT
\end{equation}
for all $x \in \HH^1(k_v,A[2])$.
This, with Remark~\ref{R:4-torsion}, implies
that $q_v$ takes values in $\frac12\Z/\Z$ (instead of just $\frac14\Z/\Z$)
if and only if $c_{\calT,v}=0$.
Thus 
\begin{center}
$Q$ takes values in $\frac12\Z/\Z$ $\iff$
$c_\calT \in \Sha^1(k,A[2])$.
\end{center}
For an example with $c_\calT \in \Sel_2 A-\Sha^1(k,A[2])$,
and another example with $0 \ne c_\calT \in \Sha^1(k,A[2])$,
see \cite{Poonen-Rains-selfcup-preprint}*{Example~3.12}.
\end{remark}

\begin{remark}
\label{R:adjustment}
Suppose that we are considering a family of abelian varieties
with a systematic subgroup $G$ of $\Sel_\lambda A$ coming from rational points
(e.g., a family of elliptic curves with rational $2$-torsion).
Let $X$ be the image of $G$ in
$\sideset{}{'}\prod_{v \in \Omega} 
(\HH^1(k_v,A[\lambda]),\HH^1(\calO_v,A[\lambda]))$.
Then our model for $\Sel_\lambda A$ should be that its image
in $X^\perp/X$ is an intersection of random maximal isotropic subgroups.
In particular, the size of $\Sel_\lambda A/\Sha^1(k,A[\lambda])$
should be distributed as $\#X$ times the size of the random intersection.
\end{remark}

\begin{example}
Suppose that $\Char k \ne p$
and $A$ is an elliptic curve $E \colon y^2=f(x)$.
The theta divisor $\Theta$ on $E$ is the identity point with multiplicity $1$.
Let $\LL=\OO(p\Theta)$.
Then $\lambda$ is $E \stackrel{p}\to E$,
and $\Sel_\lambda A$ is the $p$-Selmer group $\Sel_p E$.
Moreover, $\Sha^1(k,E[p])=0$ 
by Proposition~\ref{P:Sha^1=0}\eqref{I:Sha of E[p]}.
Thus Theorem~\ref{T:Selmer is intersection}
identifies $\Sel_p E$ as an intersection of 
two maximal isotropic subspaces in an $\F_p$-vector space.
Moreover, the values of quadratic form on that space are killed by $p$,
even when $p=2$, since $c_\calT=0$ in~\eqref{E:self cup product}.
In particular, $\dim \Sel_p E$
should be expected to be distributed according to $X_{\Sel_p}$,
with the adjustment given by Remark~\ref{R:adjustment} when necessary
for the family at hand.
This is evidence for
Conjecture~\ref{C:conjecture for Sel_p E}\eqref{I:distribution for Sel_p E}.
The rest of Conjecture~\ref{C:conjecture for Sel_p E},
concerning moments, is plausible given
Proposition~\ref{P:moments}\eqref{I:moments}.
\end{example}

\begin{example}
The same reasoning applies
to the $p$-Selmer group of the Jacobian of a hyperelliptic curve
$y^2=f(x)$ over a global field of characteristic not $2$
in the following cases:
\begin{itemize}
\item $f$ is separable of degree $2g+1$, and $p$ is arbitrary;
\item $f$ is separable of degree $2g+2$, and $p$ is odd.
\end{itemize}
(Use Propositions~\ref{P:Sha^1=0}\eqref{I:Sha of A[2]} 
and~\ref{P:Sha usually 0}, and Remark~\ref{R:extension of Sha usually 0}.)
This suggests Conjecture~\ref{C:conjecture for odd degree hyperelliptic}.
\end{example}

\begin{example}
\label{E:Sel_2 of genus 2 curve}
Consider $y^2=f(x)$ over a global field $k$ of characteristic not $2$ 
with $\deg f=2g+2$ for {\em even} $g \ge 2$.
Proposition~\ref{P:Sha usually 0} 
and Remark~\ref{R:extension of Sha usually 0} 
show that $\Sha^1(k,A[p])$ is $0$ with probability $1$ 
for each $p \ne \Char k$.
But the Hilbert irreducibility theorem shows that $c_\calT \ne 0$ 
with probability~$1$, 
so Remarks \ref{R:4-torsion again} and~\ref{R:4-torsion}
suggest that $\dim \Sel_2 A$ now has the distribution $X_{\Sel_2}+1$.
This suggests Conjecture~\ref{C:even genus}.

In the analogous situation with $g$ odd, it is less clear what to
predict for $\Sel_2 A$: 
using techniques in \cite{Poonen-Stoll1999}*{\S9.2}
one can show that the probability that $c_{\calW_1}$
lies in $\Sel_2 A$ is strictly between $0$ and $1$,
and the existence of this element may invalidate the random model.
\end{example}

\section{Relation to heuristics for $\Sha$ and rank}
\label{S:Delaunay}

The Hilbert irreducibility theorem shows that asymptotically $100\%$
of elliptic curves (ordered by na\"{\i}ve height) have $E(\Q)[p]=0$.
(For much stronger results, see \cite{Duke1997} and \cite{Jones2010}.)
So for statistical purposes, when letting $E$ run over all elliptic curves,
we may ignore contributions of torsion to the $p$-Selmer group.

In analogy with the Cohen-Lenstra heuristics~\cite{Cohen-Lenstra1983}, 
Delaunay has formulated a conjecture describing the distribution
of Shafarevich-Tate groups of random elliptic curves over $\Q$.
We now recall his conjectures for $\dim_{\F_p} \Sha[p]$.
For each prime $p$ and $r \in \Z_{\ge 0}$, 
let $X_{\Sha[p],r}$ be a random variable
taking values in $2\Z_{\ge 0}$ such that 
\[
	\Prob\left(X_{\Sha[p],r}=2n\right) \; = \; p^{-n(2r+2n-1)} 
		\frac{\prod_{i=n+1}^\infty (1-p^{-(2r+2i-1)})}
		{\prod_{i=1}^n (1-p^{-2i})}.
\]
The following conjecture is as in 
\cite{Delaunay2001}*{Example~F and Heuristic Assumption}, 
with the correction that $u/2$ in the Heuristic Assumption
is replaced by $u$ (his $u$ is our $r$).
This correction was suggested explicitly in~\cite{Delaunay2007}*{\S3.2}
for rank~$1$, and it seems natural to make the correction
for higher rank too.

\begin{conjecture}[Delaunay]
\label{C:Delaunay}
Let $r,n\in \Z_{\ge 0}$.
If $E$ ranges over elliptic curves over $\Q$ of rank $r$, up to isomorphism,
ordered by conductor, then the fraction with $\dim_{\F_p} \Sha(E)[p]=2n$
equals $\Prob\left(X_{\Sha[p],r}=2n\right)$.
\end{conjecture}

If the ``rank'' $r$ itself is a random variable $R$,
viewed as a prior distribution, 
then the distribution of $\dim \Sel_p E$ 
should be given by $R + X_{\Sha[p],R}$.
On the other hand, 
Theorem~\ref{T:Selmer is intersection} suggests that $\dim \Sel_p E$ 
should be distributed according to $X_{\Sel_p}$.
Let $R_{\textup{conjectured}}$ be the random variable 
taking values $0$ and $1$ with probability $1/2$ each.

\begin{theorem}
\label{T:Delaunay compatibility}
For each prime $p$,
the unique $\Z_{\ge 0}$-valued random variable $R$ 
such that $X_{\Sel_p}$ and $R+X_{\Sha[p],R}$ have the same distribution
is $R_{\textup{conjectured}}$.
\end{theorem}

\begin{proof}
First we show that $R_{\textup{conjectured}}$ has the claimed property.
This follows from the following identities for $n \in \Z_{\ge 0}$:
\begin{align*}
  \Prob\left(X_{\Sel_p}=2n \right) \; 
	&= c \prod_{j=1}^{2n} \frac{p}{p^j-1} \\
	&= \frac12 \prod_{i \ge 0} (1-p^{-(2i+1)}) \cdot 
		p^{-n(2n-1)} \prod_{j=1}^{2n} (1-p^{-j})^{-1} \\
	&= \frac12 p^{-n(2n-1)} \prod_{i \ge n+1} (1-p^{-(2i-1)}) 
		\prod_{i=1}^n (1-p^{-2i})^{-1} \\
	&= \frac12 \Prob\left(X_{\Sha[p],0} = 2n \right)
\end{align*}
\begin{align*}
  \Prob \left(X_{\Sel_p}=2n+1 \right) \; 
	&= c \prod_{j=1}^{2n+1} \frac{p}{p^j-1} \\
	&= \frac12 \prod_{i \ge 0} (1-p^{-(2i+1)}) \cdot 
		p^{-n(2n+1)} \prod_{j=1}^{2n+1} (1-p^{-j})^{-1} \\
	&= \frac12 p^{-n(2n+1)} \prod_{i \ge n+1} (1-p^{-(2i+1)}) 
		\prod_{i=1}^n (1-p^{-2i})^{-1} \\
	&= \frac12 \Prob\left(X_{\Sha[p],1} = 2n \right).
\end{align*}

Next we show that any random variable $R$ 
with the property has the same distribution as $R_{\textup{conjectured}}$.
For $r \in \Z_{\ge 0}$, define a function $f_r \colon \Z_{\ge 0} \to \R$
by $f_r(s):=\Prob(X_{\Sha[p],r} = s-r)$.
The assumption on $R$ implies that
\[
	\sum_{r=0}^\infty \Prob(R=r) f_r(s) 
	= \sum_{r=0}^\infty \Prob\left(R_{\textup{conjectured}}=r \right) f_r(s).
\]
Thus to prove that $R$ and $R_{\textup{conjectured}}$ have the same
distribution, it will suffice to prove that the functions $f_r$ 
are linearly independent in the sense that for
any sequence of real numbers $(\alpha_r)_{r \ge 0}$ 
with $\sum_{r \ge 0} |\alpha_r| < \infty$
the equality $\sum_{r=0}^\infty \alpha_r f_r = 0$ 
implies that $\alpha_r=0$ for all $r \in \Z_{\ge 0}$.
In fact, $\alpha_r=0$ by induction on $r$,
since $f_r(s)=0$ for all $s>r$, and $f(r,r)>0$.
\end{proof}

\section*{Acknowledgements} 

We thank Christophe Delaunay, Benedict Gross, 
Robert Guralnick, and Karl Rubin for comments.

\end{document}